\documentclass[11 pt]{amsart}
\usepackage{amssymb}
\usepackage{amsmath}
\usepackage{amsfonts}
\usepackage{graphicx}
\usepackage{amsthm}
\usepackage{enumerate}
\usepackage[mathscr]{eucal}
\usepackage{verbatim}

 \theoremstyle{plain}
\newtheorem{theorem}{Theorem}[section]
\newtheorem{lemma}[theorem]{Lemma}

\numberwithin{equation}{section}
\theoremstyle{plain}

\theoremstyle{remark}

\def\bbR{{\mathbb {R}}}

\begin{document}

\date{July, 2013}

\title
{Application of a Fourier restriction theorem to 
certain 
families of projections in $\bbR^3$
}

\author[]
{Daniel Oberlin and Richard Oberlin}

\address
{Daniel  Oberlin \\
Department of Mathematics \\ Florida State University \\
 Tallahassee, FL 32306}
\email{oberlin@math.fsu.edu}

\address
{Richard Oberlin \\
Department of Mathematics \\ Florida State University \\
 Tallahassee, FL 32306}
\email{roberlin@math.fsu.edu}


\thanks{D.O. was supported in part by NSF Grant DMS-1160680
and R.O. was supported in part by NSF Grant DMS-1068523.}

\begin{abstract}
We use a restriction theorem for Fourier transforms of fractal measures to study projections onto families of planes in $\bbR^3$ whose normal directions form nondegenerate curves. 
\end{abstract}

\maketitle

\section{Introduction and statement of results}

Suppose that $\gamma:[0,1]\rightarrow S^2$ is $C^{(2)}$. Following 
K. F\"assler and T. Orponen \cite{FO} we say that $\gamma$ is {\it nondegenerate} if  
\begin{equation*}\label{pd}
\text{span}\,\{\gamma (t),\gamma ' (t), \gamma '' (t)\}=\bbR^3 ,\ t\in [0,1].
\end{equation*}
Let $\pi_t$ be the orthogonal projection of $\bbR^3$ onto the plane $\gamma (t )^\perp$
and let $B\subset\bbR^3$ be a compact set with Hausdorff dimension $\dim(B)=\alpha$.
One of the problems treated in \cite{FO} is to say something about the 
dimension of $\pi_t (B)$ for generic $t\in[0,1]$.  F\"assler and Orponen prove that 

(a) if $\alpha \le 1$ then
%
$$
\dim (\pi_t (B))=\dim (B)
$$

for almost all $t\in [0,1]$, and

(b) if $\alpha >1$ then there exists $\sigma =\sigma (\alpha )>1$ such that the packing 

dimension of $\pi_t (B)$ exceeds $\sigma$ for almost all $t$. 

\noindent In a subsequent paper, \cite{O}, Orponen considers the particular $\gamma$ given by
\begin{equation}\label{p-1}
\gamma (t)=\frac{1}{\sqrt 2}(\cos t ,\sin t ,1)
\end{equation}
and establishes the analog of (b) for Hausdorff dimension. 
(We mention that, in addition to other interesting results of a similar nature, the papers \cite{FO} and \cite{O} provide a nice account of the history of these problems.) 
The purpose of this note is to prove the following result:

\begin{theorem}\label{projection}
%
%
With notation as above, 
suppose that $B$ is a compact subset 
of $\bbR^3$ and that $\dim (B)=\alpha \ge 1$. Then, for almost all $t\in[0,1]$,  
\begin{multline}\label{p0}
\ \ \ \ \dim \big(\pi_t (B)\big)\ge 3\alpha /4 \text{ if }1\le \alpha \le 2 \text{ and }\\
\dim \big(\pi_t (B)\big)\ge \alpha -1/2 \text{ if }2\le \alpha\le 3.\ \ \ \ 
\end{multline}
\end{theorem}
\noindent The proof uses the potential-theoretic method introduced in \cite{K}, which we approach using 
the Fourier transform as in \cite{F}. In the model case \eqref{p-1} a critical role 
in the proof of Theorem \ref{projection} 
is played by the 
following result of Erdo\u{g}an, which can be extracted from \cite{E}:

\begin{theorem}\label{BE}
Suppose that $\mu$ is a nonnegative and compactly supported Borel probability measure on $\bbR^3$ satisfying 
\begin{equation}\label{r2}
\mu \big( B(x,r)\big)\le c\, r^\alpha
\end{equation}
for $x\in\bbR^3$ and $r>0$. If $\alpha '<\alpha$ then there is $C$ 
(depending only on $c, \alpha'$, and the diameter of the support of $\mu$)
such that
\begin{equation*}
\int_0^{2\pi} \int_{1/2}^1|\hat \mu \big( R\,\rho\, (\cos t ,\sin t ,1)\big)|^2 d\rho \,dt\le 
C\,R^{-\beta (\alpha ' )}, R\ge 1, 
\end{equation*}
where 
$\beta (\alpha ' )=\alpha ' /2$ if $1\le \alpha \le 2$ and
$\beta (\alpha ' )=\alpha ' -1$ if $2\le \alpha \le 3$.
\end{theorem}
\noindent To prove Theorem \ref{projection} we require the following generalization of Theorem \ref{BE}:
\begin{theorem}\label{restriction}
Suppose $\phi:[-1/2,1/2]\rightarrow \bbR$ is $C^{(2)}$ and satisfies 
\begin{equation*}
M/2 \le |\phi ' (t)|\leq M,\ \ |\phi '' (t)|\ge m> 0
\end{equation*}
for $t\in[0,1]$ and for positive constants $M$ and $m$. Suppose that $\mu$ is a nonnegative and compactly supported Borel probability measure on $\bbR^3$ satisfying 
\begin{equation*}
\mu \big( B(x,r)\big)\le c\, r^\alpha
\end{equation*}
for $x\in\bbR^3$ and $r>0$. If $\alpha '<\alpha$ then there is $C$ 
(depending only on $m,M,c,\alpha'$, and the diameter of the support of $\mu$)
such that 
%
%
%
\begin{equation}\label{r3}
\int_{-1/2}^{1/2} \int_{1/2}^1|\hat \mu \big( R\,\rho \,(t,\phi (t), 1)\big)|^2 d\rho \,dt
\le C\,R^{-\beta (\alpha ' )}, R>1,
\end{equation}
where 
$\beta (\alpha ' )=\alpha ' /2$ if $1\le \alpha \le 2$ and
$\beta (\alpha ' )=\alpha ' -1$ if  $2\le \alpha \le 3$.

\end{theorem}
\noindent (For a similar generalization of Wolff's result in \cite{W} on decay of circular means of Fourier transforms of measures on $\bbR^2$, see \cite{EO}.)

This note is organized as follows: \S 2 contains the proof of Theorem \ref{projection}, \S 3 contains the proof of Theorem \ref{restriction}, and \S 4 contains the proof of a technical lemma used in \S 3.

\section{Proof of Theorem \ref{projection}}
Suppose $ \alpha ' <\tilde \alpha <\alpha $ so that we can find a 
probability measure $\mu$ on $B$ satisfying
\begin{multline}\label{p1}
\ \ \ \ \ \ \ \ \mu \big(B(x,r)\big)\le C\,r^{\tilde \alpha}, \, x\in\bbR^3 ,\, r>0 \text{ and } \\
\int_{\bbR^3} \frac{|\hat \mu (\xi)|^2}{|\xi |^{3-\tilde \alpha}} \,d\xi <\infty .\ \ \ \ \ \ \ \ \ \ \ \ \ \ \ 
\end{multline}
Write $\pi_t (\mu )$ for the measure which is the push-forward of $\mu$ onto 
$\pi_t (\bbR^3 )$  under the projection $\pi_t$. 
For a function $f$ on $ \gamma (t )^\perp$  we have 
\begin{equation*}
\int_{ \gamma (t )^\perp}f\,d\pi_t (\mu )=
\int_{\bbR^3} f\big(x-[x\cdot \gamma (t )]\gamma (t )\big)\, d\mu (x)
\end{equation*}
so for $\xi \in \gamma (t )^\perp$ we have 
\begin{equation*}
\widehat{\pi_t (\mu )}(\xi )=\int_{\bbR^3}
e^{-2\pi i \langle \xi , (x-[x\cdot \gamma (t )]\gamma(t )\rangle}\, d\mu (x)=\hat\mu (\xi ).
\end{equation*}
To establish \eqref{p0} it is therefore enough to show that for each $t_0 \in (0,1)$ 
there is some closed interval $I=I_{t_0}$ containing $t_0$ in its interior such that  
\begin{equation}\label{p3}
\int_{I}\int_{ \gamma (t )^\perp}\frac{|\hat{\mu }(\xi )|^2}
{|\xi |^{2-\tau (\alpha '  )}}\,d\xi\,dt <\infty  ,
\end{equation}
where
\begin{equation}\label{p.1}
\tau (\alpha '  )=3\alpha '  /4 \text{  if } 1\le \alpha \le 2 ,\ \ \tau (\alpha '  )=\alpha '  -1/2 \text{   if  } 2<\alpha \le 3 .
\end{equation}
Without loss of generality we assume that $\gamma$ is parametrized by arclength. 
Let $ u(t)=\gamma' (t)$, $ v(t)=\gamma (t)\times \gamma ' (t)$. We parametrize 
$\gamma (t)^\perp$ by 
$$
(u,v)\mapsto u\cdot u (t)+v\cdot v(t)
$$
and part of $\bbR^3$ by 
\begin{equation}\label{p3.5}
(t,u,v)\mapsto u\cdot u (t)+v\cdot v(t) . 
\end{equation}
%
%
%
If we now parametrize $(u,v)$-space by polar coordinates 
$$
u=r\sin\theta ,\ v=r\cos\theta
$$
then \eqref{p3} becomes 
\begin{equation}\label{p4}
\int_0^{2\pi}\int_I \int_0^\infty
\frac{|\hat\mu \big(r(\sin\theta \, u(t)+\cos \theta \, v (t)\, )\big)|^2}{r^{1-\tau (\alpha '  )}}\, dr\,dt\,d\theta 
<\infty .
\end{equation}
To establish \eqref{p4} for every $\alpha '  <\bar \alpha$ it is enough to show that 
\begin{equation}\label{p4.25}
\int_0^{2\pi}\int_I \int_R^{2R}
\frac{|\hat\mu \big(r(\sin\theta \, u(t)+\cos \theta \, v (t)\, )\big)|^2}{r^{1-\tau (\alpha '  )}}
\, dr\,dt\,d\theta \le C(I, \alpha '  ),\  R\ge R(I)>1
\end{equation}
for every $\alpha '  <\bar \alpha$.
We will focus, without loss of generality, on the part of the integral 
in \eqref{p4.25} 
corresponding to the range $0\le\theta \le \pi /2$. We write
\begin{equation}\label{4.3}
\eta (R) =R^{-\alpha ' /4} \text { if } 1\le \alpha \le 2,\ \ \eta (R)=R^{-1/2} \text{  if } 2<\alpha\le 3 
\end {equation}
and then split the integral: 
\begin{multline}\label{p4.5}
\int_0^{\eta (R)}\int_I \int_R^{2R}
\frac{|\hat\mu \big(r(\sin\theta \, u(t)+\cos \theta \, v (t)\, )\big)|^2}{r^{1-\tau (\alpha '  )}}
\, dr\,dt\,d\theta + \\
\int_{\eta (R)}^{\pi /2}\int_I \int_R^{2R}
\frac{|\hat\mu \big(r(\sin\theta \, u(t)+\cos \theta \, v (t)\, )\big)|^2}{r^{1-\tau (\alpha '  )}}
\, dr\,dt\,d\theta \doteq \\
\mathcal{I}_1 +\mathcal{I}_2 .
\end{multline}

We begin with the second of these and will use the change of variable \eqref{p3.5}. 
The Jacobian factor $J=J(t,u,v)$ associated with \eqref{p3.5} is 
\begin{equation*}
|\det\big( u(t),   v(t), u \cdot u ' (t)+v\cdot v ' (t)\big)|=
|\langle u \cdot u ' (t)+v \cdot v ' (t) ,\gamma (t)\rangle |,
\end{equation*}
where the last equality follows because $ u(t)\times  v(t)=\pm \gamma (t)$. Since
$u(t)\perp \gamma (t)$ implies
$$
\langle  u ' (t),\gamma (t)\rangle = -\langle  u (t), \gamma ' (t)\rangle
$$
and similarly for $ v(t)$, we see that 
\begin{equation}\label{p5}
J=|u|. 
\end{equation}
To use \eqref {p5} we need some information about the multiplicity  of the change of variables \eqref{p3.5}. To obtain this information we will impose a first restriction on the size of the interval 
$I=I_{t_0}$. (When we deal with with the second integral in \eqref{p4.5} we will need to impose further restrictions on $I$.)
 Fix $t_0$ and choose coordinates for $\bbR^3$ so that 
$\gamma (t_0 )= (0,0,1)$ and $\gamma' (t_0 )=(1,0,0)$ and then write 
$\gamma=(\gamma_1 ,\gamma_2 ,\gamma_3 )$. Let $\bar\gamma (t)$ be the curve in 
$\bbR^2$ given, in a neighborhood of $t_0$, by 
\begin{equation}\label{5.1}
\bar\gamma (t)= \Big(\frac{\gamma_1 (t)}{\gamma_3 (t)},\frac{\gamma_2 (t)}{\gamma_3 (t)}\Big)
\end{equation}
and let $\tilde\gamma$ be the curve in $\bbR^3$ given by 
\begin{equation}\label{5.2}
\tilde\gamma (t)= 
\gamma (t)/\gamma_3 (t) =\big(\bar\gamma (t);1\big).
\end{equation}
We will need that fact that
if $\kappa (\bar\gamma ;t)$ is 
the curvature of $\bar\gamma$ at $t_0$ then 
\begin{equation*}
\kappa (\bar\gamma ; t_0)= |\det \big(\gamma (t_0 ),\gamma' (t_0), \gamma'' (t_0 )\big)|>0,
\end{equation*}
where the inequality is a consequence of the non degeneracy of $\gamma$. To see the
equality we begin by computing
\begin{equation*}
\bar\gamma ' =\Big(
\frac{\gamma_1 ' \gamma_3 -\gamma_1 \gamma_3 '}{\gamma_3^2},\
\frac{\gamma_2 ' \gamma_3 -\gamma_2 \gamma_3 '}{\gamma_3^2}
\Big)
\end{equation*}
and
\begin{equation*}
\bar\gamma ''=\Big(
\frac{\gamma_3^2 (\gamma_1 '' \gamma_3 -\gamma_1 \gamma_3 '')-
2\gamma_3 \gamma_3 ' (\gamma_1 ' \gamma_3 -\gamma_1 \gamma_3 ')}
{\gamma_3^4},
\frac{\gamma_3^2 (\gamma_2 '' \gamma_3 -\gamma_2 \gamma_3 '')-
2\gamma_3 \gamma_3 ' (\gamma_2 ' \gamma_3 -\gamma_2 \gamma_3 ')}
{\gamma_3^4}
\Big).
\end{equation*}
When $t=t_0$ we have 
\begin{equation*}
\bar\gamma (t_0)=(1,0),\ \bar\gamma '' (t_0 )=\big( \gamma_1 '' (t_0 ), \gamma_2 '' (t_0 )\big).
\end{equation*}
Thus $\kappa (\bar\gamma ;t_0 )=|\gamma_2 '' (t_0 )|= |\det \big(\gamma (t_0 ),\gamma' (t_0), \gamma'' (t_0 )\big)|$ as desired. Now choose $I$ small enough so that 
\begin{equation*}
\kappa (\bar\gamma ;t)>0 ,\ t\in I .
\end{equation*}
After possibly shrinking $I$ again one sees that
if $t_1 ,t_2 ,t_3 \in I$, then the vectors 
$\{\tilde\gamma (t_1 ),\tilde\gamma (t_2 ),\tilde\gamma (t_3 )\}$ are linearly independent. 
Since 
$$
\tilde\gamma (t_i)\perp \big(u \cdot u (t_i )+v \cdot v (t_i )\big),
$$
it follows that 
\eqref{p3.5} is at most three-to-one on $I\times (\bbR^2 \sim\{0\})$. 
Therefore, with 
$$
\xi =r\big(\sin\theta \, u(t)+\cos \theta \, v (t)\big)
$$
and using \eqref{p5} to write $J=|u|=|r\sin \theta |$ we have 
\begin{multline}\label{p6.25}
\mathcal {I}_2 =\int_{\eta (R)}^{\pi /2}\int_I \int_R^{2R}
\frac{|\hat\mu \big(r(\sin\theta \, u(t)+\cos \theta \, v (t)\,)\big)|^2}{r^{1-\tau (\alpha '  )}}
\, dr\,dt\,d\theta \lesssim \\
\frac{1}{\eta (R)}\int_{R\le |\xi |\le 2R}\frac{|\hat \mu (\xi )|^2}{|\xi |^{3-\tau (\alpha '  )}}\, d\xi \le
C\, R^{\alpha '  -\bar \alpha}
\end{multline}
by \eqref{p1}, \eqref{p.1}, and \eqref{4.3}.

We now obtain a similar estimate for the term $\mathcal{I}_1$. Lemma 3.2 from \cite{FO} states that the function $ v(t)=\gamma (t)\times \gamma ' (t)$ satisfies the same 
hypotheses as $\gamma (t)$:  
\begin{equation*}
\text{span}\{ v (t), v ' (t),  v '' (t)\}=\bbR^3 ,\ t\in [0,1].
\end{equation*}
We proceed as above, beginning by choosing coordinates for $\bbR^3$ so that $v(t_0 )=(0,0,1)$,
$v' (t_0 )=(1,0,0)$.  
It follows that if $\bar v$ and $\tilde v$ are defined as in \eqref{5.1} and \eqref{5.2} but with 
$ v(t)$ in place of $\gamma (t)$, then $\kappa (\bar v ;t_0) >0$. For small $\theta$ we will also need the perturbations of $\bar v$ and $\tilde v$ given by taking  
\begin{equation}\label{p7}
v_\theta (t)= \cos\theta\,  v(t) +\sin\theta\,  u(t) 
\end{equation}
instead of $\gamma$ in \eqref{5.1} and \eqref{5.2}. Using $\kappa (\bar v ;t_0) >0$ 
we choose $\theta_0 >0$ such that 
%
$\kappa (\bar v_\theta ;t_0) >0$
%
for $0\le\theta\le\theta_0$. We then further restrict 
$\theta_0$ and
the interval $I=I_{t_0}$ so that
\begin{equation}\label{p7.1}
|v_\theta (t)-(0,0,1)|\le 1/10 \text{ for } t\in I,\ 0\le\theta \le \theta_0
\end{equation}
and, for some $\tilde m >0$, we have 
\begin{equation*}
\kappa (\bar v_\theta ;t) \ge \tilde m \text{ for } 0\le\theta\le\theta_0 , \  t\in I.
\end{equation*}
After a suitable linear change of coordinates in $\bbR^2 $ 
we choose positive numbers $M$ and $m$ 
such that 
(after possibly diminishing $\theta_0$ and $I$)
the curves $\{\tilde v_\theta ( t): t\in I\}$, $0\le\theta\le\theta_0$, can be written as
\begin{equation*}
\{\big(\tilde t,\phi_\theta (\tilde t) ,1\big):\tilde  t\in[-1/2,1/2]\}
\end{equation*}
with 
\begin{equation*}
M/2 \le |\phi_\theta ' (\tilde t)|\leq M,\ \ |\phi_\theta '' (\tilde t)|\ge m > 0,
 \ \ -1/2\le\tilde t\le 1/2 ,\ \ 0\le\theta\le\theta_0 .
\end{equation*}
It follows from Theorem \ref{restriction} that for $0\le\theta\le\theta_0$ we have
\begin{equation*}
\int_I \int_{9R/10}^{22R/10}
|\hat\mu \big(r(\tilde v_\theta (t)\,)\big)|^2\, dr\,dt 
\le C\, R^{1-\beta (\alpha '  )} .
\end{equation*}
So if $R(I)$ is chosen to have $\eta \big(R(I)\big)=\theta_0$ then it follows from \eqref{p.1},
\eqref{4.3}, \eqref{p7}, \eqref{p7.1}, and the definition of $\beta (\alpha '  )$ that
\begin{equation*}
\mathcal{I}_1 =
\int_0^{\eta (R)}\int_I \int_R^{2R}
\frac{|\hat\mu \big(r(\sin\theta \, u(t)+\cos \theta \, v (t)\big)|^2}{r^{1-\tau (\alpha '  )}}
\, dr\,dt \, d\theta \le C(I,\alpha '  ),\ R\ge R(I) .
\end{equation*}
With \eqref{p6.25} and \eqref{p4.5} this gives \eqref{p4.25} and therefore completes the proof of Theorem \ref{projection}.

\section{Proof of Theorem \ref{restriction}}

For $2< \alpha \le 3$, \eqref{r3} follows directly from Theorem 1 in \cite{E}. 
For  $1\le \alpha\le 2$, the proof is an adaptation of ideas from \cite{E} and \cite{W}. Specifically, we will write 
%
\begin{multline*}
\sigma (\rho, t)=\rho \,\big(t,\phi (t), 1\big), \\
\Gamma_R =\{R\,\sigma (t,\rho ):-1/2 \le t\le 1/2 ,\, 1/2\le \rho\le 1 \}, \\
\Gamma_{R,\delta}=\Gamma_R +B(0,R^\delta ),\ R\ge 2,\ \delta >0
\end{multline*}
%
and, with $\mu$ as in Theorem \ref{restriction}, we will 
show that \eqref{r3} follows from the estimate 
\begin{equation}\label{main3}
\int_{\Gamma_{R,\delta}}|\widehat{\mu}(y)|^2\,dy \lesssim R^{2-\alpha /2 +2\delta},\ 
0<\delta \le 1.
\end{equation}
We will then adapt a bilinear argument from \cite{E} to prove \eqref{main3}.
(Throughout this proof the constants implied by the symbol $\lesssim$
can be chosen to depend only $\delta$ and on the parameters mentioned for $C$ in the statement of Theorem \ref{projection}.)

So, arguing as in \cite{W}, if $\kappa\in C^{\infty}_c (\bbR^3 )$ is equal to $1$ on the support of $\mu$, then
\begin{multline}\label{main4}
\int_{-1/2}^{1/2} \int_{1/2}^1|\hat \mu \big( R\,\sigma (\rho ,t)\big)|^2 d\rho \,dt
 =
 \int_{-1/2}^{1/2} \int_{1/2}^1\Big|
 \int \widehat\kappa\big( R\,\sigma (\rho ,t)-y\big)\,\widehat{\mu}(y)dy\Big|^2 d\rho \,dt
\lesssim \\
\int \int_{-1/2}^{1/2} \int_{1/2}^1 \big| \widehat{\kappa}\big( R\,\sigma (\rho ,t)-y\big)
\big|\,d\rho\, dt\ |\widehat{\mu}(y)|^2 \,dy .
\end{multline}
Let $\epsilon =\alpha -\alpha'$. Choose a large $p_1$ such that if $y\notin \Gamma_{R}+B(0,R^{p_1})$ then 
$\text{dist}(y,\Gamma_{R,\epsilon /4})\ge |y|/2$. Choose a large $p_2$ such that if 
$y\in \Gamma_{R}+B(0,R^{p_1})$ then $|y|\le R^{p_2}$. Finally, choose a large $K$ such that 
$(K-4)\epsilon /4 \ge3 p_2 $. 
If $y=(y_1 ,y_2 ,y_3 )$, then
\begin{multline*}
\int_{-1/2}^{1/2} \int_{1/2}^1 \big| \widehat{\kappa}(R\,\sigma (\rho ,t )-y)\big|\,d\rho  \,dt \lesssim
\int_{-1/2}^{1/2} \int_{1/2}^1 \frac{1}{\big(1+|R\,\sigma (\rho ,t)-y|\big)^{K}}\, d\rho\,dt \lesssim \\
\frac{1}{\big(1+\text{dist} (\Gamma_R ,y)\,\big)^{K-4}}\
 \int_{1/2}^1  \int_{-1/2}^{1/2} \frac{1}{\big(1+|R\,\rho\,\phi (t)-y_2 |\big)^2} \,dt \ 
  \frac{1}{\big(1+|R\,\rho-y_3 |\big)^2} \,d\rho.
\end{multline*}
Estimating the last two integrals (we use the hypothesized lower bound 
on $\phi'$), we see from \eqref{main4} that 
\begin{equation}\label{main4.5}
\int_{-1/2}^{1/2}
\int_{1/2}^1 |\widehat{\mu}\big(R\,\sigma (\rho ,t)\big)|^2 d\rho \, dt \lesssim
\frac{1}{R^2 }\int\frac{|\widehat{\mu}(y)|^2}{\big(1+\text{dist} (\Gamma_R ,y)\,\big)^{K-4}}\, dy .
\end{equation}
Now
\begin{multline*}
\int\frac{|\widehat{\mu}(y)|^2}{\big(1+\text{dist} (\Gamma_R ,y)\,\big)^{K-4}}\, dy
=\int_{\Gamma_{R,\epsilon /4}}+\int_{B(0,R^{p_2})\sim{\Gamma_{R,\epsilon /4}} }+
\int _{\{|y|\ge R^{p_2}\}}.
\end{multline*}
The first integral, the principal term, is $\lesssim  R^{2-\alpha' /2}$ by \eqref{main3}. 
Since $y\notin \Gamma_{R,\epsilon /4}$ implies 
$\text{dist}(\Gamma_R ,y)\ge R^{\epsilon /4}$, the second integral is $\lesssim 1$ by the 
fact that 

\noindent$(K-4)\epsilon /4 \ge3 p_2 $. Since $|y|\ge R^{p_2}$ implies 
$y\notin \Gamma_{R,p_1}$ and so implies $\text{dist}(\Gamma_R ,y)\ge |y|/2$, the last integral is also $\lesssim 1$.
Thus, given \eqref{main3},  \eqref{r3} follows from \eqref{main4.5}.

Turning to the proof of \eqref{main3}, we note that by duality 
(and the fact that $\mu$ is finite)
it is enough to 
suppose that $f$, satisfying $\|f\|_2 =1$, is supported on 
$\Gamma_{R,\delta}$ and then to establish the estimate
\begin{equation}\label{main5}
\int |\widehat{f}(y)|^2\, d\mu (y) \lesssim
 \,R^{2-\alpha /2 +2\delta}.
\end{equation}
The argument we will give differs from the proof of Theorem 5 in \cite{E} only in certain technical details. But,
because those details are not always obvious, we will give the complete proof.  

For $y\in\bbR^3$, write $y'$ for a point $R\sigma (\rho ' , t')$ 
($\rho ' \in [1/2 ,1],\, t' \in [-1/2,1/2]$)
on the surface $\Gamma_R$ which minimizes $\text{dist}(y, \Gamma_R )$.
For a dyadic interval $I\subset  [-1/2,1/2]$, define
\begin{equation*}
\Gamma_{R,\delta,I}=\{y\in \Gamma_{R,\delta}: t' \in I\},\, f_I =
f\cdot\chi_{\Gamma_{R,\delta,I}}\, .
\end{equation*}
For dyadic intervals $I,J \subset  [-1/2,1/2]$, we write $I\sim J$ if $I$ and $J$ have the same length and are not adjacent but have adjacent parent intervals. The decomposition 
\begin{equation}\label{main6}
[-1/2,1/2]\times [-1/2,1/2]={\bigcup_{n\ge 2}} \Big(
\bigcup _{\substack{{|I|=|J|=2^{-n}}\\{I\sim J}}}(I\times J)
\Big)
\end{equation}
leads to 
\begin{equation}\label{main7}
\int |\widehat{f}(y)|^2\, d\mu (y)\leq \sum_{n\ge 2 }\ \ \sum_{\substack{{|I|=|J|=2^{-n}}\\{I\sim J}}}
\int|\widehat{f_I} (y)\widehat{f_J}(y)|\,d\mu (y).
\end{equation}
Truncating \eqref{main6} and \eqref{main7} gives 
\begin{multline}\label{main8}
\int |\widehat{f}(y)|^2\, d\mu (y)\leq \\
\sum_{4\leq 2^n \leq 
R^{1/2}}\ \ \sum_{\substack{{|I|=|J|=2^{-n}}\\{I\sim J}}}
\int|\widehat{f_I} (y)\widehat{f_J}(y)|\,d\mu (y) +\sum_{I\in \mathcal I} \int|\widehat{f_I} (y)|^2\,d\mu (y),
\end{multline}
where $\mathcal I$ is a finitely overlapping set of dyadic intervals $I$ with $|I| \approx R^{-1/2}$.

To estimate the integrals on the right hand side of \eqref{main8}, we begin 
with two geometric observations. The first of these is that 
it follows from the hypotheses on $\phi$ that
if $I\subset [-1/2,1/2]$ is an interval with length $\ell$, then 
\begin{equation*}
\Gamma_{R,I}\doteq \{ R\,\sigma (\rho ,t): t\in I,\, 1/2\le \rho\le 1\}
\end{equation*}
is contained in a rectangle $D$ with side lengths $\lesssim R ,R\ell  ,R\ell ^2$, which we will abbreviate by saying that $D$ is an $R\times (R\ell ) \times (R\ell ^2)$ rectangle. 
%
%
Secondly, we observe that if $\ell\gtrsim R^{-1/2}$, then an $R^\delta$ neighborhood of an 
$R\times (R\ell ) \times (R\ell ^2)$ rectangle is contained in an $R^{1+\delta}\times (R^{1+\delta}\ell ) \times (R^{1+\delta}\ell ^2)$ rectangle.
It follows that if $I$ has length $2^{-n}\gtrsim R^{-1/2}$, then the support of $f_I$ is contained in a rectangle 
$D$ with dimensions $R^{1+\delta}\times (R^{1+\delta}2^{-n} ) \times (R^{1+\delta}2^{-2n})$.

The next lemma is part of Lemma 4.1 in \cite{E}. To state it, we introduce some notation: 
$\phi$ is a nonnegative Schwartz function such that $\phi (x)=1$ for $x$ in the unit cube $Q$,  $\phi (x)=0$ if $x\notin 2Q$,
and, for each $M>0$, 
\begin{equation*}
|\widehat{\phi}|\leq C_M \sum_{j=1}^\infty 2^{-Mj}\chi_{2^j Q}.
\end{equation*}
For a rectangle $D\subset\bbR^3$, $\phi_D$ will stand for $\phi\circ b$, where $b$ is an affine mapping
which takes $D$ onto $Q$. If $D$ is a rectangle with dimensions 
$a_1 \times a_2 \times a_3$, then a dual rectangle of $D$ is any rectangle with the same axis directions and with dimensions $a_1^{-1}\times a_2^{-1}\times a_3^{-1}$.
\begin{lemma}
Suppose 
$1\le\alpha\le2$ and 
that $\mu$ is a non-negative Borel measure on $\bbR^3$ satisfying 
\eqref{r2}. Suppose $D$ is a rectangle with dimensions $R_1 \times R_2 \times R_3$, where $R_3\lesssim R_2 \lesssim R_1\lesssim R$,
and let $D_{\text{dual}}$ be the dual of $D$ centered at the origin. Then, 
if $\widetilde{\mu}(E)=\mu (-E)$,
\begin{equation}\label{main9}
(\widetilde{\mu}\ast |\widehat{\phi _D}|)(y)\lesssim R_2^{2-\alpha}R_1,\, y\in\bbR^3 
\end{equation}
and, if $K\gtrsim 1,\, y_0 \in \bbR^3$, then
\begin{equation}\label{main10}
\int_{K\cdot D_{\text{dual}}}(\widetilde{\mu}\ast |\widehat{\phi _D}|)(y_0 +y)\,dy\lesssim K^{\alpha}R_2^{1-\alpha}R_3^{-1}.
\end{equation}
\end{lemma}
\noindent Now if $I\in\mathcal I$ and $\text{supp}f_I \subset D$ as above, the identity 
$\widehat{f_I}=\widehat{f_I}\ast \widehat{\phi _D}$ implies that
\begin{equation*}
|\widehat{f_I}|\leq (|\widehat{f_I}|^2 \ast |\widehat{\phi _D}|)^{1/2}\|\widehat{\phi_D}\|_1^{1/2}\lesssim
(|\widehat{f_I}|^2 \ast |\widehat{\phi _D}|)^{1/2} 
\end{equation*}
and so 
\begin{multline}\label{main11}
\int|\widehat{f_I} (y)|^2\,d\mu (y)\lesssim
 \int  (|\widehat{f_I}|^2 \ast |\widehat{\phi _D}|)(y)\,d\mu (y)
= \\
\int |\widehat{f_I}(y)|^2 (\widetilde{\mu}\ast |\widehat{\phi_D} |)(-y)\,dy
\lesssim 
\|f_I \|_2^2 \, R^{2-\alpha /2 +2\delta},
\end{multline}
where the last inequality follows from \eqref{main9}, the fact that $D$ has dimensions 
$R^{1+\delta}\times R^{1/2 +\delta}\times R^\delta$ since $2^{-n}\approx R^{-1/2}$,
and the inequalities $1\le\alpha \le 2$. Thus the estimate 
\begin{equation}\label{main12}
\sum_{I\in \mathcal I} \int|\widehat{f_I} (y)|^2\,d\mu (y)\lesssim 
R^{2-\alpha /2 +2\delta}
\sum_{I\in \mathcal I}\|f_I\|_2^2 \lesssim
R^{2-\alpha /2 +2\delta}
\end{equation}
follows from $\|f\|_2 =1$ and the finite overlap of the intervals $I\in\mathcal I$
(which implies finite overlap for the supports of the $f_I ,I\in\mathcal I$).

To bound the principal term of the right hand side of \eqref{main8}, fix $n$ with $4\leq 2^n \leq R^{1/2}$
and a pair $I,J$ of dyadic intervals with $|I|=|J|=2^{-n}$ and $I\sim J$. Since $I\sim J$, the support of 
$f_I \ast f_J$ is contained in a rectangle $D$ with dimensions 
$R^{1+\delta}\times (R^{1+\delta}2^{-n} )\times (R^{1+\delta}2^{-2n})$. For later reference, let $u,v,w$ be unit vectors in the directions of the sides of $D$ 
with $u$ parallel to the longest side and $w$ parallel to the shortest side. As in \eqref{main11},
\begin{multline}\label{main13}
\int|\widehat{f_I} (y)\widehat{f_J} (y)|\,d\mu (y)\lesssim 
\int  (|\widehat{f_I}\,\widehat{f_J}| \ast |\widehat{\phi _D}|)(y)\,d\mu (y)
= \\
\int |\widehat{f_I}(y)\widehat{f_J}(y)|\, (\widetilde{\mu}\ast |\widehat{\phi_D} |)(-y)\,dy.
\end{multline}
Now tile $\bbR^3$ with rectangles $P$ having exact dimensions 
$(C2^{-2n})\times (C2^{-n}) \times C$ 
for some large $C>0$ to be chosen later
and having shortest side in the direction of $u$ and longest side in the direction of $w$. 
Let $\psi$ be a fixed nonnegative Schwartz function satisfying $1\leq \psi (y)\leq 2$ if $y\in Q$, 
$\widehat{\psi}(x)=0$ if $x\notin Q$, and 
\begin{equation}\label{main13.5}
\psi \leq C_M \sum_{j=1}^\infty 2^{-Mj} \chi_{2^j Q}.
\end{equation}
Since $\sum_P \psi_P^3 \approx 1$,
it follows from \eqref{main13} that if $f_{I,P}$ is defined by 
\begin{equation*}
\widehat{f_{I,P}}=\psi_P \cdot\widehat{f_I}
\end{equation*}
then 
\begin{multline}\label{main14}
\int|\widehat{f_I} (y)\widehat{f_J} (y)|\,d\mu (y)\lesssim \\
\sum_P \Big(\int |\widehat{f_{I,P}}(y)\widehat{f_{J,P}}(y)|^2\,dy\Big)^{1/2}
\Big(\int \big| (\widetilde{\mu}\ast|\widehat{\phi_D}|)(-y)\psi_P (y)\big|
^2\,dy\Big)^{1/2}.
\end{multline}

To estimate the first integral in this sum, we begin by noting that the support of $f_{I,P}$ is contained in $\text{supp}(f_I )+P_{\text{dual}}$, where $P_{\text{dual}}$ is a rectangle dual to $P$ and centered at the origin.
Let $\widetilde I$ be the interval with the same center as $I$ but lengthened by $2^{-n}/10$ 
and let $\widetilde J$ be defined similarly. 
Since $I\sim J$, it follows that $\text{dist}(\widetilde{I} ,\widetilde{J})\geq 2^{-n}/2$.
Now the support of $f_I$ is contained in $\Gamma_{R,I}+B(0,R^\delta )$
and $P_{\text{dual}}$ has dimensions $( 2^{2n} C^{-1})\times ( 2^{n} C^{-1})\times C^{-1}$ 
and side in the direction of $v$ at an angle 
$\lesssim2^{-n}$ to any of the tangents to the curve $\big(t,\phi (t)\big)$ 
for $t\in \widetilde{I}$ (or $t\in\widetilde{J}$). Recalling that $2^n \lesssim R^{1/2}$, one can
check that, if $C$ is large enough, 
\begin{equation*}
\text{supp}(f_{I,P})\subset \Gamma_{R,\widetilde{I}}+B(0,CR^\delta )
\end{equation*}
and, similarly,
\begin{equation*}
\text{supp}(f_{J,P})\subset \Gamma_{R,\widetilde{J}}+B(0,CR^\delta ).
\end{equation*}
The next lemma follows from Lemma \ref{second} in \S 4 by scaling:

\begin{lemma}\label{intersection lemma} 
Suppose that the closed intervals $\tilde I , \tilde J \subset [0,1]$ satisfy  
$\text{dist}\,(\tilde I ,\tilde J )\geq c \,2^{-n}$.
Then, for $\delta >0$ and $x\in\bbR^3$, there is the following estimate for the three-dimensional Lebesgue measure of the intersection of translates of neighborhoods of $\Gamma_{R,\tilde I}$ and $\Gamma_{R, \tilde J}$:
\begin{equation*}
\big| \, x+\Gamma_{R,\tilde I}+B(0,C R^\delta )\ \cap\  \Gamma_{R,\tilde J}+B(0,CR^\delta ) \,
\big| \lesssim R^{1+2\delta} 2^n .
\end{equation*}
%
\end{lemma}

\noindent It follows from Lemma \ref{intersection lemma} that for $x\in\bbR^3$ we have
\begin{equation}\label{main15}
\big| \,x+\text{supp}(f_{I,P})\ \cap\ \text{supp}(f_{J,P})\,\big|\lesssim R^{1+2\delta}2^n  .
\end{equation}
Now 
\begin{equation*}
\int |\widehat{f_{I,P}}(y)\widehat{f_{J,P}}(y)|^2\,dy =
\int|\widetilde{f_{I,P}}\ast {f_{J,P}}(x)|^2\, dx
\end{equation*}
and 
\begin{multline*}
|\widetilde{f_{I,P}}\ast {f_{J,P}}(x)|\leq\int |f_{I,P}(w-x)\,f_{J,P}(w)|\,dw \leq \\
|x+\text{supp}(f_{I,P})\cap \text{supp}(f_{J,P})|^{1/2}\, \big(|\widetilde{f_{I,P}}|^2 \ast |{f_{J,P}}|^2 (x)\big)^{1/2}.
\end{multline*}
Thus, by \eqref{main15},
\begin{multline}\label{main16}
\Big(\int |\widehat{f_{I,P}}(y)\widehat{f_{J,P}}(y)|^2\,dy \Big)^{1/2}\lesssim
R^{1/2+\delta} 2^{n/2}\Big(\int |\widetilde{f_{I,P}}|^2 \ast |{f_{J,P}}|^2 (x)\,dx \Big)^{1/2} = \\
R^{1/2 +\delta} 2^{n/2}\|f_{I,P}\|_2 \|f_{J,P}\|_2 .
\end{multline}

To estimate the second integral in the sum \eqref{main14} we use \eqref{main13.5} to observe that
\begin{equation*}
\psi_P \lesssim \sum_{j=1}^{\infty}2^{-Mj}\chi_{2^j P}.
\end{equation*}
Thus
\begin{equation*}
\int (\widetilde{\mu}\ast|\widehat{\phi_D}|)(-y)\psi_P (y) \,dy \lesssim
 \sum_{j=1}^{\infty}2^{-Mj}\int_{2^j P} (\widetilde{\mu}\ast|\widehat{\phi_D}|)(-y) \,dy .
\end{equation*}
Noting that $2^j P \subset y_P +K D_{\text{dual}}$  for some $K\lesssim R^{1+\delta}2^{-2n+j}$ 
and some 

\noindent $y_P \in \bbR^3$, we apply \eqref{main10} to obtain 
\begin{multline*}
\int (\widetilde{\mu}\ast|\widehat{\phi_D}|)(-y)\psi_P (y) \,dy \lesssim \\
\sum_{j=1}^{\infty}2^{-Mj}( R^{1+\delta}2^{-2n+j})^\alpha (R^{1+\delta}2^{-n})^{1-\alpha }(R^{1+\delta}2^{-2n})^{-1}
\lesssim 2^{-n(\alpha -1)}.
\end{multline*}
Since 
\begin{equation*}
(\widetilde{\mu}\ast|\widehat{\phi_D}|)(-y) \lesssim (R^{1+\delta}2^{-n})^{2-\alpha }\,R^{1+\delta}
\end{equation*}
by \eqref{main9} and since $\psi_P (y)\lesssim 1$, it follows that 
\begin{equation}\label{main19}
\Big(\int \big((\widetilde{\mu}\ast|\widehat{\phi_D}|)(-y)\psi_P (y)\big)^2 \,dy \Big)^{1/2}\lesssim
R^{(1+\delta )(3-\alpha )/2}\, 2^{-n/2}.
\end{equation}

Now \eqref{main16} and \eqref{main19}
imply, by \eqref{main14}, that
\begin{equation*}
\int|\widehat{f_I} (y)\widehat{f_J} (y)|\,d\mu (y)\lesssim
R^{(1+\delta )(2-\alpha /2)}\Big(\sum_P \|f_{I,P}\|_2^2 \Big)^{1/2} \Big(\sum_P \|f_{J,P}\|_2^2 \Big)^{1/2}.
\end{equation*}
Since
\begin{equation*}
\sum_P \|\widehat{f_{I,P}}\|_2^2 =
\int |\widehat{f_I}(y)|^2 \sum_P |\psi_P (y)|^2 \, dy,
\end{equation*}
it follows from $\sum_P \psi_P ^2 \lesssim 1$ that
\begin{equation*}
\int|\widehat{f_I} (y)\widehat{f_J} (y)|\,d\mu (y)\lesssim
R^{(1+\delta )(2-\alpha /2)}\|f_I \|_2 \|f_J \|_2 .
\end{equation*}
Thus 
\begin{multline}\label{main21}
\sum_{\substack{{|I|=|J|=2^{-n}}\\{I\sim J}}} \int|\widehat{f_I} (y)\widehat{f_J}(y)|\,d\mu (y)\lesssim \\
R^{(1+\delta )(2-\alpha /2)}\sum_{\substack{{|I|=|J|=2^{-n}}\\{I\sim J}}}
\|f_I \|_2 \|f_J \|_2 \lesssim \\
R^{(1+\delta )(2-\alpha /2)}\|f\|_2^2 .
\end{multline}
Now \eqref{main5} follows from \eqref{main8}, \eqref{main12}, \eqref{main21}, and the fact that the first sum in \eqref{main8} has $\lesssim \log R$ terms.
This completes the proof of Theorem \ref{restriction}.

\section {Two lemmas}

As mentioned in \S 3, Lemma \ref{intersection lemma} follows from Lemma \ref{second}
below. The proof of Lemma \ref{second} will use the following fact:

\begin{lemma}\label{first}
Suppose $\phi_1 ,\phi_2$ are functions on $[c,d]$ with $|\phi_1 '(u_1 )-\phi_2' (u_2 )|\ge a>0$ for all $u_1 ,u_2 \in [c,d]$ and  
$|\phi_i ' |\leq b$. For $\delta' >0$ let 
\begin{equation*}
C_{i,\delta'}=\{\big(u,\phi_i (u)\big):u\in [c,d]\}+B(0,\delta' ).
\end{equation*}
Then
\begin{equation}\label{c}
|C_{1,\delta'}\cap C_{2,\delta'}|\le 16(1+b)^2 \delta'^2 /a  .
\end{equation}
\end{lemma}
\begin{proof}
To prove the lemma
we begin by noting that we may the extend the $\phi_i$ so that they are defined on 
$[c-\delta' ,d+\delta' ]$ and satisfy the lemma's hypotheses on this larger interval.
We may also assume that the intersection in \eqref{c} is nonempty and then choose 
$u_0 \in [c,d]$ with 
\begin{equation*}
|\phi_1 (u_0 ) -\phi_2 (u_0 )|< (2+2b)\delta' .
\end{equation*}
Because of the assumptions on the $\phi_i$ it follows that 
\begin{equation}\label{d}
\text{if }|u-u_0 |\ge(4+4b) \delta' /a \text{ then } |\phi_2 (u) -\phi_1 (u)|\ge (2+2b)\delta' .
\end{equation}
Now assume that $(x_1 ,x_2 )$ is in the intersection in \eqref{c} and so 
\begin{equation*}
|(x_1 ,x_2 )-(u_i ,\phi_i (u_i ))|<\delta' ,\ i=1,2
\end{equation*}
for some $u_1 ,u_2 \in [c,d]$.
Then 
\begin{equation}\label{e}
|x_1 -u_i |,\ |x_2 -\phi_i (u_i )|<\delta' , \ i=1,2 .
\end{equation}
Now 
$$
|x_2 -\phi_i (x_1)|\le |x_2 -\phi_i (u_i )|+| \phi_i (u_i )-\phi_i (x_1) | < (1+b) \delta' 
$$
by \eqref{e} and $|\phi_i ' |\le b$
and so $|\phi_1 (x_1 )-\phi_2 (x_1 )| <(2+2b)\delta'$. 
Thus \eqref{d} shows that
\begin{equation}\label{f}
|x_1 -u_0 |< (4+4b)\delta' /a .
\end{equation}
%
Since $|x_2 -\phi_1 (x_1)|< (1+b)\delta'$, it follows 
from \eqref{f} (and the fact that $(x_1 ,x_2 )$ is a generic point of the intersection in \eqref{c}) that  \eqref{c} holds, proving the lemma. 
\end{proof}

\begin{lemma}\label{second}
Suppose the real-valued function $\phi$ on $[-1/2,1/2]$ satisfies estimates 
$|\phi (t)|,\,|\phi ' (t)|\le M$, $|\phi '' (t)|\ge m>0$ for $-1/2\le t\le 1/2$. For 
$I\subset [-1/2,1/2]$ 
and $0<\delta' <1$ define
\begin{equation*}
\Sigma_{I,\delta'}\doteq \{\rho \big(t,\phi (t),1\big): t\in I,\, 1/2\le \rho \le 1\}+B(0,\delta' ).
\end{equation*}
There is a positive constant $C$ depending only on $M$ and $m$ such that 
if $I,J\subset [-1/2,1/2]$ are $\ell$-separated subintervals of $[0,1]$ and $x\in\bbR^3$ then 
\begin{equation}\label{L1}
\big|(x+\Sigma_{I,\delta'})\cap \Sigma_{J,\delta}|\le C\delta'^2 /\ell .
\end{equation}
\end{lemma}
\begin{proof}
We can extend $\phi$ to $[-20,20]$ with estimates 
$|\phi (t)|,\,|\phi ' (t)|\le M'$, $|\phi '' (t)|\ge m'>0$ for $-3\le t\le 3$ and with $M'$, $m'$ depending only on $M$ and $m$. 
Our strategy will be to estimate the two-dimensional Lebesgue measure of certain sections 
of $(x+\Sigma_{I,\delta'})\cap \Sigma_{J,\delta'}$.
Recall the notation $\sigma (t,\rho )=\rho \big(t,\phi (t),1\big)$ and assume that 
$y,y' \in (x+\Sigma_{I,\delta'})\cap \Sigma_{J,\delta'}$. Then there are $\rho, t, \tilde\rho , \tilde t$ with $t\in I,\, \tilde t\in J$ such that 
\begin{equation}\label{L2}
|y-x-\sigma (t,\rho)|,\, |y-\sigma (\tilde t,\tilde\rho )|<\delta' 
\end{equation}
and there are 
$\rho ', t', \tilde\rho ', \tilde t'$ with $t'\in I,\, \tilde t '\in J$ such that 
\begin{equation}\label{L3}
|y'-x-\sigma (t',\rho')|,\, |y'-\sigma (\tilde t ',\tilde\rho ')|<\delta' .
\end{equation}
Write $x=(x_1 ,x_2 ,x_3 )$ and similarly for $y$ and $y'$. We are interested in the two-dimensional measure of the section defined by \lq\lq third coordinate = $c$" - in fact,
\eqref{L1} will follow when we show that the two-dimensional Lebesgue measure 
of this section is $\le C\delta'^2 /\ell$ - 
and so we assume that $y_3 =y'_3 =c$. It follows from \eqref{L2} and \eqref{L3} that 
\begin{equation*}
|c-x_3-\rho |,|c-x_3 -\rho ' |<\delta'
\end{equation*}
and thus that $|\rho -\rho ' |<2\delta'$. Similarly, $|\tilde\rho-\tilde\rho ' |<2\delta'$. Thus 
there are fixed $\rho, \tilde\rho$ such that 
if $y' \in (x+\Sigma_{I,\delta'})\cap \Sigma_{J,\delta'}$ then there are 
$t'\in I,\, \tilde t '\in J$ such that 
\begin{equation}\label{L4}
|y'-x-\sigma (t',\rho)|,\, |y'-\sigma (\tilde t ',\tilde\rho )|<C\delta' .
\end{equation}
where $C$ denotes, as it always will in this proof, a positive constant depending only on $m$ and $M$. It follows from \eqref{L4} that 
\begin{equation}\label{L5}
|(y_1 ',y_2 ' )-(x_1 ,x_2 )-\big(\rho\, t' ,\rho\, \phi (t' )\big)|,\,
|(y_1 ',y_2 ' )-\big(\tilde\rho\, \tilde t ' ,\tilde\rho\, \phi (\tilde t ' )\big)|<C \delta' .
\end{equation}
Define 
\begin{multline*}
\phi_1 (u)=x_2 +\rho\, \phi \big( (u-x_1 )/\rho\big), \phi_2 (u)=\tilde\rho\, \phi (u/\tilde\rho ), 
\\
[c_1 ,d_1 ]=x_1 +\rho I, [c_2 ,d_2 ]=\tilde\rho J, [c,d]=[c_1 ,d_1 ]\cap [c_2 ,d_2 ].
\end{multline*}
Then \eqref{L5} implies that 
%
\begin{multline}\label{L6}
|(y_1' ,y_2 ')-\big(u_1 ,\phi_1 (u_1 )\big)|,\,
|(y_1' ,y_2 ')-\big(u_2 ,\phi_2 (u_2 )\big)|<C\delta' \\
\text{ for some }
u_1 \in [c_1 ,d_1 ],u_2 \in [c_2 ,d_2 ].
\end{multline}
Recall that our goal is to show that the set of all $(y_1' ,y_2 ')$ for which \eqref{L6} holds has two-dimensional Lebesgue measure $\le C\delta'^2 /\ell$. In the case $u_1 ,u_2 \in [c,d]$ this follows from Lemma \ref{first} with $a=m' \ell$. 
(The derivative separation requirement in Lemma \ref{first} is a consequence of our hypothesis $|\phi '' |\ge m'$.)
Of the remaining cases, 
$u_1\in [c_1,d_1 ], u_2 \notin [c_1 ,d_1]$ is typical: allowing $C$ to increase from line to line,
$|u_1 -u_2 |<C\delta'$ follows from 
\eqref{L6}. So, since $u_1\in [c_1,d_1 ], u_2 \notin [c_1 ,d_1]$, we have in succession that 
\begin{multline*}
\text{dist}\big(u_2 ,\{c_1 ,d_1 \}\big)\le C\delta' ,\   \ 
\text{dist}\big(y_1 ' ,\{c_1 ,d_1 \}\big)\le C\delta' , \\
\text{ and }\text{dist}\big(y_2 ' ,\{\phi_2 (c_1 ) ,\phi_2 (d_1 )\}\big)\le C\delta' .
\end{multline*}
(Note to the very careful reader: the extension of $\phi$ to the interval $[-20,20]$ guarantees that both $\phi_1$ and $\phi_2$ are defined on an interval which contains both $[c_1 ,d_1 ]$ and $[c_2 ,d_2 ]$.)
Thus the set of all $(y_1' ,y_2 ')$ for which \eqref{L6} holds with 
$u_1\in [c_1,d_1 ], u_2 \notin [c_1 ,d_1]$ has two-dimensional Lebesgue measure 
$\le C\delta'^2 $. This completes the proof of Lemma \ref{second}.

\end{proof}

\end{document}